\documentclass{amsart}

\usepackage{amssymb,latexsym}
\usepackage[latin1]{inputenc} 
\usepackage[dvips]{graphicx,color,psfrag}

\numberwithin{equation}{section}
\numberwithin{figure}{section}
\numberwithin{table}{section}

\theoremstyle{plain}
\newtheorem{prop}{Proposition}[section]
\newtheorem{lemma}[prop]{Lemma}
\newtheorem{theorem}[prop]{Theorem}

\theoremstyle{definition}
\newtheorem{defin}{Definition}[section]

\newenvironment{proofoft}[1]{{\em Proof of Theorem #1. }}{$\Box$ \vspace{1em}}
\newenvironment{proofofl}[1]{{\em Proof of Lemma #1. }}{$\Box$ \vspace{1em}}

\newcommand{\E}{{\mathbf E}}

\newcommand{\Z}{{\mathbb Z}}
\newcommand{\Zd}{{\mathbb Z}^d}

\newcommand{\RR}{{\mathbb R}}
\newcommand{\N}{{\mathbb N}}

\newcommand{\Remark}{\em Remark: \rm}

\newcommand{\CL}{{}_L\,{}^{\!\!}C}
\newcommand{\dptn}{\,{}^{\!\!}}%Approx -0.3 pt

\def\P{{\mathbf P}}

\begin{document}

%Article info
\title[The critical contact process in a varying environment dies out]
{The critical contact process in a randomly evolving environment dies out}
%\date{\today}

%Author info
\author[J.~Steif]{Jeffrey E.~Steif$^{1,2}$}
\address[J.~Steif and M.~Warfheimer]{Department of Mathematical Sciences, Chalmers University of Technology and University of Gothenburg, SE-41296 Gothenburg, Sweden}
\email[J.~Steif]{steif@chalmers.se}
\urladdr{http://www.math.chalmers.se/~steif}
\thanks{$^1$Research partially supported by the Swedish Natural Science
Research Council.}

\author[M.~Warfheimer]{Marcus Warfheimer$^2$}
\email{marcus.warfheimer@gmail.com}
\urladdr{http://www.math.chalmers.se/~warfheim}
\thanks{$^2$Research partially supported by the Göran Gustafsson Foundation for Research in Natural Sciences and Medicine.}

%AMS info

\begin{abstract}
Bezuidenhout and Grimmett proved that the critical contact process dies out. Here, we generalize the result to the so called contact process in a random evolving environment (CPREE), introduced by Erik Broman. This process is a generalization of the contact process where the recovery rate can vary between two values. The rate which it chooses is determined by a background process, which evolves independently at different sites. As for the contact process, we can similarly define a critical value in terms of survival for this process. In this paper we prove that this definition is independent of how we start the background process, that finite and infinite survival (meaning nontriviality of the upper invariant measure) are equivalent and finally that the process dies out at criticality. 
\end{abstract}

\maketitle

\section{Introduction and main results}

The contact process, introduced by Harris \cite{Harris1}, is a simple model for the spread of an infection on a lattice. The state at a certain time is described by a configuration, $\eta \in \{0,1\}^{\Zd}$, where $\eta(x)=0$ means that the individual at location $x$ is healthy and $\eta(x)=1$ means it is infected. The model is such that infected people recover at rate $1$ and healthy people are infected with a rate proportional to the number of infected neighbors. In more mathematical language, the contact process is a Markov process, $\{\eta_t\}_{t\geq 0}$, with state space $\{0,1\}^{\Zd}$ where the configuration changes its state at site $x\in\Zd$ as follows: 
\[
\begin{array}{lllll}
\eta \to \eta_x   \qquad & \text{with rate} &\qquad 1 \qquad &\text{if} \quad & \eta(x)=1 \\
\eta \to \eta_x   \qquad & \text{with rate} &\qquad \lambda\displaystyle\sum_{y\sim x} \eta(y) \qquad &\text{if} \quad & \eta(x)=0,
\end{array}
\]
where $y\sim x$ means that $x$ and $y$ are neighbors, 
\[
\eta_x(y)=
\begin{cases}
\eta(y) & \text{if $y\neq x$} \\
1-\eta(x) & \text{if $y=x$}
\end{cases}
\]
and $\lambda$ is a positive parameter called the infection rate. See the standard references Liggett \cite{Liggett85} and Durrett \cite{Durrett88} for how these informal rates determine a Markov process and for much on the contact process as well as other interacting particle systems. Denote the distribution of this process when it starts with the configuration $\eta$ by $\P_\lambda^\eta$. We say that the process \emph{dies out at $\lambda$} if 
\[
\P_{\lambda}^{\{0\}}[\,\eta_t= \emptyset\text{ some }t \geq 0 \,]=1;
\]
otherwise it is said to \emph{survive at $\lambda$}. Here, the initial configuration $\{0\}$ means there is a single infection at the origin and the configuration \small$\emptyset$ \normalsize means the element in $\{0,1\}^{\Zd}$ consisting of all zeros. (As usual, we identify $\{0,1\}^{\Zd}$ with subsets of $\Zd$.) Using an easy monotonicity in $\lambda$, it is natural to define the critical value
\[
\lambda_c:=\inf\{\,\lambda:\:\P_\lambda^{\{0\}}[\,\eta_t\neq\emptyset \text{ for all }t \geq 0 \,]>0 \,\}.
\]
A fundamental first question concerning this model is whether it survives when $\lambda$ is large and whether it dies out for small values of $\lambda$, i.e.$\dptn$ whether $0<\lambda_c<\infty$, and it is not very hard to show that this indeed is the case. Furthermore, since the contact process is attractive (see Liggett \cite{Liggett85} for this definition), we can define
\[
\lambda_c^\prime:=\inf\{\,\lambda:\:\bar{\nu}_\lambda\neq\delta_\emptyset\,\},
\]
where $\bar{\nu}_\lambda$ is the so called upper invariant measure, defined to be the limiting distribution starting from all 1's. A self-duality equation (see \cite{Durrett88} or \cite{Liggett85}) easily leads to $\lambda_c=\lambda_c^\prime$.  A much harder question, and one which had been open for approximately 15 years, is whether the contact process survives or dies out \emph{at} the critical value. A celebrated theorem by Bezuidenhout and Grimmett, \cite{GrimBez}, gives us the answer.
\begin{theorem}[Bezuidenhout and Grimmett]\label{paper1:GB}
The critical contact process dies out.
\end{theorem}
\noindent For a proof of this, see \cite{GrimBez} or \cite{Liggett99}. 

Note that changing $\lambda$ to $1$ and the recovery rate to $\delta$ corresponds to a trivial time scaling and so the process could have instead been defined in this way. We will denote the corresponding critical value by $\delta_c$. This should be kept in mind in what follows. 

In 1991, Bramson, Durrett and Schonmann \cite{Durrett1} introduced the contact process in a random environment, in which the recovery rates are taken to be independently and identically distributed random variables and then fixed in time. For further results concerning this model see for example, Liggett \cite{Liggett1}, Klein \cite{Extinction1} and Newman and Volchan \cite{Survival1}. Recently, Broman \cite{Broman} introduced another variant where the environment changes in time in a simple Markovian way. More precisely, he considered the Markov process, $\{(B_t,C_t)\}_{t \geq 0}$ on $\{0,1\}^{\Zd}\times\{0,1\}^{\Zd}$ described by the following rates at a site $x$:
\[
\begin{array}{ll}
\textrm{transition} \qquad & \qquad \textrm{rate} \vspace{7pt} \\ 
(0,0) \rightarrow (0,1) \qquad &  \qquad\displaystyle\sum_{y\sim x} C(y) \\
(1,0) \rightarrow (1,1) \qquad &  \qquad\displaystyle\sum_{y\sim x} C(y) \\
(0,1) \rightarrow (0,0) \qquad &  \qquad\delta_0 \\
(1,1) \rightarrow (1,0) \qquad &  \qquad\delta_1 \\
(0,0) \rightarrow (1,0) \qquad &  \qquad\gamma p \\
(0,1) \rightarrow (1,1) \qquad &  \qquad\gamma p \\
(1,0) \rightarrow (0,0) \qquad &  \qquad\gamma (1-p) \\
(1,1) \rightarrow (0,1) \qquad &  \qquad\gamma (1-p) 
\end{array}
\]
where $d \geq 1$, $\gamma, \delta_0, \delta_1>0$ with $ \delta_1 \leq\delta_0 $ and $p \in [0,1]$. In other words, at each site $x$ independently, $\{B_t(x)\}_{t\geq 0}$ is a 2-state Markov chain with infinitesimal matrix
\[
\begin{pmatrix}
-\gamma p & \gamma p \\
\gamma (1-p) & -\gamma (1-p)
\end{pmatrix}
\]
which in turn determines the recovery rate of $\{C_t(x)\}_{t\geq 0}$ in the following way. For each $t$, the recovery rate at location $x$ is $\delta_0$ or $\delta_1$ depending on whether $B_t(x)=0$ or $B_t(x)=1$. In addition, the infection rate is always taken to be the number of infected neighbors. (Actually, Broman did this on a more general graph, but here we will only consider $\Zd$.) Broman referred to $\{B_t\}_{t\geq 0}$ as \emph{the background process} and the whole process $\{(B_t,C_t)\}_{t\geq 0}$ as \emph{the contact process in a randomly evolving environment} (CPREE). Let $\{C_t^{\rho}\}_{t\geq 0}$ denote the right marginal where the initial distribution of the whole process is $\rho$. In the case where $\rho=\mu\times\nu$ we write $\{C_t^{\mu,\nu}\}_{t\geq 0}$. Furthermore, let $\P_p$ denote the measure governing the process for the parameters $p$, $\gamma$, $\delta_0$ and $\delta_1$, where $\gamma$, $\delta_0$ and $\delta_1$ are considered fixed. Also, denote the product measure with density $q\in [0,1]$ by $\pi_q$. Broman defined the critical value 
\[
p_c:=\inf\left\{ p:\, \P_p [\,C_t^{\pi_p,\{0\}}\neq \emptyset\, \forall t > 0\,] > 0 \right\} 
\]
($p_c$ is taken to be $1$ if no $p$ satisfies this) and proved that if $\delta_1<\delta_c<\delta_0$ and $\gamma > \max(2d,\delta_c-\delta_1)$, then $p_c\in (0,1)$. At the end of his paper he asked whether the critical value is affected if we vary the initial distribution of the background process. Our first result answers this question. Given $\gamma, \delta_0, \delta_1 >0$ with $\delta_1 \leq \delta_0$, $q \in [0,1]$ and $A \subseteq \Zd$ with $\vert A \vert < \infty$, define
\[
p_c(q,A):=\inf\left\{ p:\, \P_p [\,C_t^{\pi_q,A}\neq \emptyset\, \forall t > 0\,] > 0 \right\}.
\]
\begin{theorem}\label{paper1:thm1}
Given $A$,$A^\prime\subseteq\Zd$ with $|A|$, $|A^\prime|<\infty$ and $p$, $q$, $q^\prime \in [0,1]$,
\begin{equation}\label{paper1:thm1eq}
\P_p[\,C_t^{\pi_q,A} \neq \emptyset\, \forall t > 0\,]>0 \quad \iff \quad \P_p[\,C_t^{\pi_{q^\prime},A^\prime} \neq \emptyset\, \forall t > 0\,]>0.
\end{equation}
In particular, $p_c(q,A)$ is independent of both $q$ and $A$.
\end{theorem}
We will let $p_c$ denote this common value. (Recall, $p_c$ of course depends on $\gamma$, $\delta_0$ and $\delta_1$.) Also, if $\P_p[\,C_t^{\pi_q,A} \neq \emptyset\, \forall t > 0\,]>0$ holds (which we now know is independent of $q$ and $A$), we say that $\{C_t\}$ \emph{survives at $p$}; otherwise it is said to \emph{die out at $p$}.  

Later on, we will see that the process is attractive. (See Proposition~\ref{paper1:attr}.) This yields that the limiting distribution starting from all 1's exists and we will denote the limit by $\bar{\nu}_p$. Also, we will refer to this measure as the \emph{upper invariant measure}. This measure gives us another natural way to define a critical value:
\[
p_c^\prime:=\inf\{\, p: \, \bar{\nu}_p \neq \pi_p\times\delta_{\emptyset}\,\}.
\]
For general attractive systems it might or might not be the case that these definitions coincide. However, for the ordinary contact process, this is the case (due to its self-duality) and our next result shows that this is also true in our situation.
\begin{theorem}\label{paper1:thm2}
$\{C_t\}$ survives at $p$ if and only if $\bar{\nu}_p \neq \pi_p\times\delta_{\emptyset}$. In particular $p_c=p_c^\prime$.
\end{theorem} 
Our final result is a generalization of Theorem~\ref{paper1:GB}.
\begin{theorem}\label{paper1:thm3}
If $\{C_t\}$ survives at $p>0$, then there exists $\delta>0$ so that it survives at $p-\delta$. In particular, if $p_c \in (0,1]$, then the critical contact process in a randomly evolving environment dies out.
\end{theorem}

The rest of the paper is organized as follows. In Section 2, we provide some preliminaries, in Section 3, we prove Theorems~\ref{paper1:thm1} and \ref{paper1:thm2} and in Section 4, we prove Theorem~\ref{paper1:thm3}. 

\section{Some preliminaries}
\label{paper1:graphrepr}

In this section we will present the basic construction of  the CPREE via a graphical representation that is suitable for our situation. We will also prove the elementary fact that the CPREE is an attractive process. However, we will start off with some notation and basic definitions. When the initial distribution of the process is $\rho$, we will denote the distribution at time $t$ by $\rho S_p(t)$, suppressing $\gamma$, $\delta_0$ and $\delta_1$ in the notation. (Of course, $\rho$ is a probability measure on $\{0,1\}^{\Zd}\times\{0,1\}^{\Zd}$.) When $\rho$ is a product measure, $\rho=\mu\times\nu$, we will denote the process by $\{(B_t^{\mu}, C_t^{\mu,\nu})\}_{t\geq 0}$. In the case where $\mu=\delta_\beta$ and $\nu=\delta_\eta$ for some $\beta$, $\eta \in \{0,1\}^{\Zd}$, we write $\{(B_t^{\beta}, C_t^{\beta,\eta})\}_{t\geq 0}$. To simplify notation, we freely interchange between talking about elements in $\{0,1\}^{\Zd}$ and subsets of $\Zd$. For $\eta, \eta^\prime \in \{0,1\}^{\Zd}$ we write $\eta \leq \eta^\prime$ if $\eta(x)\leq \eta^\prime(x)$ $\forall x \in \Zd$. Furthermore, for $(\beta,\eta),(\beta^\prime,\eta^\prime) \in \{0,1\}^{\Zd}\times\{0,1\}^{\Zd}$ we write $(\beta,\eta)\leq (\beta^\prime,\eta^\prime)$ if both $\beta \leq \beta^\prime$ and $\eta\leq\eta^\prime$. These relations induce the concept of increasing function in the usual way.
\begin{defin}
We say that a function $f$ on $\{0,1\}^{\Zd}$ $($or $\{0,1\}^{\Zd}\times\{0,1\}^{\Zd})$ is increasing if $f(\eta)\leq f(\eta^\prime)$ $(f(\beta,\eta)\leq f(\beta^\prime,\eta^\prime))$ whenever $\eta\leq\eta^\prime$ $((\beta,\eta)\leq (\beta^\prime,\eta^\prime))$. 
\end{defin}
In our analysis we make extensive use of the concept of \emph{stochastic domination}.
\begin{defin}\label{paper1:dom}
Given two probability measures $\mu_1$ and $\mu_2$ on $\{0,1\}^{\Zd}$, we say that $\mu_1$ is stochastically dominated by $\mu_2$ if $\mu_1(f)\leq \mu_2(f)$ $\forall$ increasing continuous functions $f$ and we denote this by $\mu_1\leq\mu_2$. If $\mu_i$ is the distribution of $X_i$, $i=1,2$, we also write $X_1\leq_{D} X_2$. 
\end{defin} 
It is well known (see for example \cite{Liggett85}) that this is equivalent to the existence of random variables $X_1,X_2$ on a common probability space such that $X_1\sim \mu_1$, $X_2\sim\mu_2$ and $X_1 \leq X_2$ a.s. (The $\sim$ here means distributed according to.) Also, since we can identify $\{0,1\}^{\Zd}\times\{0,1\}^{\Zd}$ with $\{0,1\}^{\Zd\times\{0,1\}}$ we have a similar result for measures on $\{0,1\}^{\Zd}\times\{0,1\}^{\Zd}$. (Of course, stochastic domination makes sense on any space of the form $\{0,1\}^S$ where $S$ is countable.) 

Now, we turn to the graphical representation from which our process will be defined. Let $\gamma, \delta_0, \delta_1>0$ with $ \delta_1 \leq\delta_0 $ and $p \in [0,1]$ be given parameters. Let $\{e_j\}_{j=1}^d$ denote the standard basis on $\Zd$, i.e.$\dptn$ for $i$, $j\in \{\,1,\ldots,d\,\}$
\[
e_j(i)=
\begin{cases}
1 \quad \text{if $i=j$} \\
0 \quad \text{if $i\neq j$}. 
\end{cases}
\]
Define the following stochastic elements on a common probability space in such a way that they are independent:
\begin{itemize}
\item[--] $M^{b,0\to 1}=\{M^{b,0\to 1 }_t\}_{t\geq 0}$, a process with state space $\N^{\Zd}$ where each marginal independently evolves as a Poisson process with intensity $\gamma p$. (This process will correspond to the 0 to 1 flips in the background process, see below.)
\item[--] $M^{b,1\to 0}=\{M^{b,1\to 0}_t\}_{t\geq 0}$, a process with state space $\N^{\Zd}$ where each marginal independently evolves as a Poisson process with intensity $\gamma(1-p)$. (This process will correspond to the 1 to 0 flips in the background process, see below.)
\item[--] $N^{\delta_1}=\{N_t^{\delta_1}\}_{t\geq 0}$, a process with state space $\N^{\Zd}$ where each marginal independently evolves as a Poisson process with intensity $\delta_1$.
\item[--] $N^{\delta_0-\delta_1}=\{N_t^{\delta_0-\delta_1}\}_{t\geq 0}$, a process with state space $\N^{\Zd}$ where each marginal independently evolves as a Poisson process with intensity $\delta_0-\delta_1$.
\item[--] $\vec{N}^j=\{\vec{N}_t^j\}_{t \geq 0}$, $j\in \{\,\pm e_1,\ldots,\pm e_d \,\}$, independent processes with state space $\N^{\Zd}$ where each marginal independently evolves as a Poisson process with intensity 1. (We think of the points in $\vec{N}^j(x)$ as being arrows from $x$ to $x+e_j$ and will correspond to the potential spread of infection from  $x$ to $x+e_j$.)
\end{itemize}
For $s\geq 0$ and $\beta\in\{0,1\}^{\Zd}$, we will begin to define a process $B^{\beta,s}=\{B_t^{\beta,s}\}_{t\geq s}$ where for each $x\in\Zd$, $B^{\beta,s}(x)$ is a function of the arrivals of $M^{b, 0\to 1}(x)$ and $M^{b,1\to 0}(x)$ in $[s,\infty)$. Assume for example that $\beta(x)=0$; the case when $\beta(x)=1$ can be handled in a similar fashion. We then define
\[
\begin{array}{lll}
B_t^{\beta,s}(x)&=0, \quad &  \quad s\leq t<T_1 \\
B_t^{\beta,s}(x)&=1, \quad &  \quad  T_1\leq t<T_2\\
B_t^{\beta,s}(x)&=0, \quad &  \quad  T_2\leq t<T_3 \\
B_t^{\beta,s}(x)&=1, \quad &  \quad  T_3\leq t<T_4 \\
&\vdots
\end{array}
\]
where $T_1$ is the first arrival time of $M^{b,0\to 1}(x)$ after $s$, $T_2$ is the first arrival time of $M^{b,1\to 0}(x)$  after $T_1$, $T_3$ is the first arrival time of $M^{b,0\to 1}(x)$  after $T_2$, $T_4$ is the first arrival time of $M^{b,1\to 0}(x)$  after $T_3$ and so forth. In words, the points in $M^{b,0\to 1}$ are the times at which the background process switches to $1$ (had it been in state $0$) and similarily for $M^{b,1\to 0}$. Note importantly, we have all the processes $B^{\beta,s}$, as $\beta$ and $s$ vary, defined on the same probability space. 

Given $B^{\beta,s}$, $N^{\delta_1}$ and $N^{\delta_0-\delta_1}$, define  $X^{\beta,s}=\{X_t^{\beta,s}\}_{t\geq s}$, a point process on $\Zd\times [s,\infty)$, in the following way: 
\[
\begin{split}
X^{\beta,s}&=\{\,(x,t)\in\Zd\times[s,\infty):\,(x,t)\in N^{\delta_1}\text{ or } \\
&\qquad (x,t)\in N^{\delta_0-\delta_1} \text{ and } B_t^{\beta,s}(x)=0 \,\}
\end{split}
\]
In words, for each site $x$, we choose points in $[s,\infty)$ from $N^{\delta_1}(x)$  when the background process is in state $1$ and from the union of $N^{\delta_1}(x)$ and $N^{\delta_0-\delta_1}(x)$ when the background process is in state $0$.
\begin{defin}\label{paper1:paths}
Given space-time points $(x,s)$ and $(y,t)$ with $t>s$ and $\beta\in\{0,1\}^{\Zd}$, we say that there is a $\beta$-active path from $(x,s)$ to $(y,t)$ if there is a sequence of times $s=s_0 < s_1 < \ldots < s_m < s_{m+1}=t$ and space points $x=x_0$, $x_1,\ldots,$ $x_m=y$ so that for $i=1,\ldots,m$, there is an arrow from $x_{i-1}$ to $x_i$ at time $s_i$ and there are no points in $X^{\beta,s}$ on the vertical segments $\{x_i\}\times (s_{i},s_{i+1})$, $i=0,\ldots,m$.
\end{defin}
\noindent\Remark Note importantly, that both $B^{\beta,s}$ and the existence of a $\beta$-active path from $(x,s)$ to $(y,t)$ are measurable with respect to the Poisson processes after time $s$ and hence are independent of everything in the Poisson processes up to that time. The reason that these objects are introduced for $s>0$ is that they are useful objects to which the original process can be usefully compared as will be done in the proof of Theorem~\ref{paper1:thm3}.

To define the process $\{(B_t^\beta,C_t^{\beta,\eta})\}_{t\geq 0}$ for a given initial configuration $(\beta,\eta) \in \{0,1\}^{\Zd}\times \{0,1\}^{\Zd}$, we let $B_t^\beta=B_t^{\beta,0}$ and
\[
\begin{split}
C_t^{\beta,\eta}&=\{\,y\in\Zd:\:\text{for some $x\in\Zd$ with $\eta(x)=1$, }\\ 
&\qquad \text{there is a $\beta$-active path from $(x,0)$ to $(y,t)$}\,\}.
\end{split}
\]
This is our formal definition of the CPREE. Note as $\beta$ and $\eta$ vary, we have all the processes $\{(B_t^\beta,C_t^{\beta,\eta})\}_{t\geq 0}$ defined on the same probability space. 

Having defined $\{(B_t,C_t)\}_{t \geq 0}$ with initial configuration $(\beta,\eta)$, it is a simple matter to extend the definition to an arbitrary initial distribution $\rho$. Just add to our probability space, independently of all the random variables already defined, two random variables on $\{0,1\}^{\Zd}$ with joint distribution $\rho$. We will denote the probability measure governing all these variables by $\P_p$, suppressing $\gamma$, $\delta_0$ and $\delta_1$ in the notation. 

The first easy fact about the CPREE we will show is that it is an attractive process. 
\begin{prop}
\label{paper1:attr}
$(B_t,C_t)$ satisfies the attractivity condition:
\begin{equation}\label{paper1:attrcond}
\rho \leq \sigma \quad \implies \quad \rho S_p(t)\leq \sigma S_p(t) \quad \forall t>0.
\end{equation}
\end{prop}
\begin{proof}
It is standard that \eqref{paper1:attrcond} is equivalent to $(\delta_\beta\times\delta_\eta)S_p(t)$ being stochastically increasing in $(\beta,\eta)$ for all $t\geq 0$. However, it is immediate from the construction that if $\beta_1\leq\beta_2$ and $\eta_1\leq\eta_2$, then for all $t\geq 0$
\[
B_t^{\beta_1}\leq B_t^{\beta_2}
\]
and
\[
C_t^{\beta_1,\eta_1}\leq C_t^{\beta_2,\eta_2}.
\]
This gives the stochastic domination (with an explicit coupling).
\end{proof} 

\section{Proofs of Theorems \ref{paper1:thm1} and \ref{paper1:thm2}}

Recall, given $\gamma$, $\delta_0$, $\delta_1 >0$ with $\delta_1 \leq \delta_0$ and $q \in [0,1]$ we have defined
\[
p_c(q,A):=\inf\left\{ p:\, \P_p [\,C_t^{\pi_q,A}\neq \emptyset\, \forall t > 0\,] > 0 \right\} 
\]
where $A \subseteq \Zd$, $\vert A \vert < \infty$, and $\pi_q$ denotes product measure with density $q$.

\medskip

\noindent\begin{proofoft}{\ref{paper1:thm1}}
We will prove the statements:
\begin{itemize}
\item[--] For all $A\subseteq\Zd$ with $|A|<\infty$ and $p$, $q\in [0,1]$,
\begin{equation}\label{paper1:stat1}
\P_p[\,C_t^{\pi_q,A} \neq \emptyset\, \forall t > 0\,]>0 \quad \iff \quad \P_p[\,C_t^{\pi_q,\{0\}} \neq \emptyset\, \forall t > 0\,]>0.
\end{equation}
\item[--] For all $p\in [0,1]$, 
\begin{equation}\label{paper1:stat2}
\P_p[\,C_t^{\emptyset,\{0\}} \neq \emptyset\, \forall t > 0\,]>0 \quad \iff \quad \P_p[\,C_t^{\Zd,\{0\}} \neq \emptyset\, \forall t > 0\,]>0.
\end{equation}
\end{itemize}
Combining these two will yield the statement in Theorem~\ref{paper1:thm1}.
For \eqref{paper1:stat1}, the left implication follows from translation invariance and the right implication follows easily from the additivity property of the process meaning
\[
C_t^{\beta,A\cup B}=C_t^{\beta,A}\cup C_t^{\beta, B} \quad \forall A,B \subseteq \Zd, \:\forall \beta\in\{0,1\}^{\Zd}.
\]
To prove \eqref{paper1:stat2}, observe that the right implication is immediate from Proposition~\ref{paper1:attr} and so we assume $\P_p [\,C_t^{\Zd,\{0\}}\neq \emptyset\, \forall t > 0\,]>0$. Define
\[
\varphi_t(x)=1_{\{B_t^{\emptyset}(x)=B_t^{\Zd}(x)\}} \quad x\in \Zd,\: t\geq 0.
\]
(Recall this is well defined since $\{B_t^{\emptyset}\}_{t\geq 0}$ and $\{B_t^{\Zd}\}_{t\geq 0}$ are defined on the same probability space.) Note that $\varphi_t$ has the property that for each site independently, after an exponentially distributed time with mean $\frac{1}{\gamma}$, the process flips to one and stays there. Therefore we have $\P_p[\,\varphi_t(x)=1\,]=1-e^{-\gamma t}$. For $A\subseteq \Zd$, define $\{\tilde{C}_t^A\}_{t\geq 0}$ from the graphical representation in the same way as $\{C_t^{\,\cdot\,,A}\}_{t\geq 0}$ except that all recoveries are ignored. This is what is usually called the Richardson model, see Durrett \cite{Durrett88}.
\begin{lemma}\label{paper1:BN}
$\P_p[\,\tilde{C}_t^{\{0\}}\subseteq \varphi_t \:, \forall t\geq n\,]\to 1$ as $n\to\infty$. 
\end{lemma}
\begin{proof}
Let $I_n=\{-n^2, \ldots,n^2\}^d$ and for $x\in \Zd$ define
\[
t(x)=\inf\{\,t:\,x\in \tilde{C}_t^{\{0\}}\,\}.
\]
From \cite[p.~16]{Durrett88}, we get that there are constants $c_1$,$c_2$,$c_3 \in (0,\infty)$ such that
\[
\P_p[\,t(x) < c_1\vert x \vert_\infty \,] \leq c_2 e^{-c_3\vert x\vert_\infty},
\]
where $\vert\cdot\vert_\infty$ is the $L^{\infty}$ norm. This easily gives us the estimate
\[
\P_p[\,\tilde{C}_{c_1(n+1)}^{\{0\}}\nsubseteq I_n\,]\leq P(n) e^{-c_3 n},
\]
where $P(n)$ is a polynomial in $n$, and from the Borel Cantelli lemma we can conclude 
\begin{equation}\label{paper1:one}
\P_p[\,\exists N \geq 1 \:\rm{such}\:\rm{that}\: \tilde{C}_{c_1(n+1)}^{\{0\}}\subseteq I_n\:, \forall n\geq N\, ]=1.
\end{equation}
Furthermore, independence gives
\[
\P_p[\,I_n\subseteq\varphi_{c_1 n}\,]=(1-e^{-\gamma c_1 n})^{(2n^2+1)^d}. 
\]
and since  
\[
\sum_{n=1}^\infty 1-(1-e^{-\gamma c_1 n})^{(2n^2+1)^d}<\infty,
\]
the Borel Cantelli lemma again yields
\begin{equation}\label{paper1:two}
\P_p[\,\exists N \geq 1 \:\rm{such}\:\rm{that}\: I_n\subseteq\varphi_{c_1 n}\:, \forall n\geq N\,]=1.
\end{equation}
Combining \eqref{paper1:one} and \eqref{paper1:two}, we obtain 
\[
\P_p[\,\exists N \geq 1 \:\rm{such}\:\rm{that}\: \tilde{C}_t^{\{0\}}\subseteq \varphi_t\:,\forall t\geq N\,]=1,
\]
as desired.
\end{proof} 

Since $C_t^{\Zd,\{0\}}\subseteq \tilde{C}_t^{\{0\}} \: \forall t\geq 0$, the claim tells us that, with probability one, after some time and thereafter, the two background processes influence $C_t^{\emptyset,\{0\}}$ and $C_t^{\Zd,\{0\}}$ in exactly the same way. Next, countable additivity gives us that for some $n\geq 1$ we have
\[
\P_p[\, \tilde{C}_t^{\{0\}}\subseteq \varphi_t\:\forall t\geq n\,,\, C_t^{\Zd,\{0\}}\neq\emptyset \: \forall t>0 \,] > 0
\]
and then that for some $m$ (depending on $n$) 
\[
\P_p[\, \tilde{C}_t^{\{0\}}\subseteq \varphi_t\:\forall t\geq n\,,\, \tilde{C}_t^{\{0\}} \subseteq [-m,m]^d\: \forall t \in [0,n], \,C_t^{\Zd,\{0\}}\neq\emptyset \: \forall t>0 \,] > 0.
\]
Denote the previous event by $A$ and define the random set
\[
U=\{\,(x,t)\in [-m,m]^d \times [0,n]: \: B_t^{\Zd}(x)=1 \, \}
\]
and let 
\[
B=\{\,\text{no arrivals in } N^{\delta_0-\delta_1} \text{ during }  U \,\}.
\]
It is clear that
\[
A \cap B \subseteq \{ \, C_t^{\emptyset,\{0\}}\neq \emptyset \: \forall t > 0  \, \}
\]
and so it remains to show that 
\[
\P_p[\,A \cap B \, ]>0.
\]
However, if we condition on $A$ and $U$, then we will not yield any information about the $N^{\delta_0-\delta_1}$ process on $U$ and so 
\[
\P_p[\,B\,| A,U\,]=e^{-(\delta_0-\delta_1)\mathcal L(U)}
\]
where $\mathcal L(U)$ is the ``length'' of $U$. This easily gives
\[
\P_p[\,B |\, A \,]>0
\]
and the proof is complete.
\end{proofoft} 

\noindent\Remark The same argument shows that strong survival does not depend on the initial distribution of the background process in the sense that
\[
\P_p[\,0\in C_t^{\emptyset,\{0\}} \text{ i.o.}\,]>0 \quad \iff \quad \P_p[\,0\in C_t^{\Zd,\{0\}}\text{ i.o.} \,]>0.
\]
This answers another question in \cite{Broman}.

Recall the definition of $p_c^\prime$ from the introduction:
\[
p_c^\prime:=\inf\{\, p: \, \bar{\nu}_p \neq \pi_p\times\delta_{\emptyset}\,\}.
\]
Here $\bar{\nu}_p=\lim_{t\to\infty}(\delta_{\Zd}\times\delta_{\Zd})S_p(t)$. (The limit exists due to Proposition~\ref{paper1:attr}.) To prove Theorem~\ref{paper1:thm2} we will use the next Lemma.
\begin{lemma}
\label{paper1:upper_conv1}
Given $p,q \in (0,1)$ with $q \geq p$ we have
\[
\lim_{t\to\infty}(\pi_q\times\delta_{\Zd})S_p(t)=\bar{\nu}_p.
\]
\end{lemma}
\begin{proof}
By simple stochastic comparison, it is enough to consider the case when $q=p$. We begin to establish the existence of that limit. Since $\pi_p$ is the stationary distribution for the background process and the right marginal always occupies less than or equal to the whole $\{0,1\}^{\Zd}$, we have 
\[
(\pi_p\times\delta_{\Zd})S_p(t)\leq \pi_p \times \delta_{\Zd} \quad \forall t > 0.
\] 
Using attractiveness and the Markov property yields
\[
(\pi_p\times\delta_{\Zd})S_p(s+t)\leq (\pi_p \times \delta_{\Zd})S_p(s) \quad \forall s, t > 0,
\]
and so the existence of the limit is clear from monotonicity. Denote this limit by $\nu_p^\prime$ and observe it is necessarily stationary. It is clear that $\nu_p^\prime\leq\bar{\nu}_p$ so we are done if $\bar{\nu}_p\leq\nu_p^\prime$. For this, note that attractiveness again gives that the map
\[
\mu \mapsto \E^\mu [f(\delta_t,\eta_t)]
\]
is increasing whenever $f$ is continuous and increasing. Using this, and the fact that any stationary distribution necessarily has as first marginal $\pi_p$, we can do the following calculation for any stationary distribution $\mu$ of $(B_t,C_t)$ and $f: \{0,1\}^{\Zd}\times\{0,1\}^{\Zd} \to \RR$ continuous and increasing:
\[
\int f \,d\mu = \E^\mu [f(\delta_t,\eta_t)] \leq \E^{\pi_p \times \delta_{\Zd}} [f(\delta_t,\eta_t)] \to \int f \,d\nu_p^\prime \quad \text{as } t\to\infty.
\]
Hence, $\mu\leq \nu_p^\prime$ and we are done.
\end{proof} 
\noindent\begin{proofoft}{\ref{paper1:thm2}}
When the initial distribution of the background process is $\pi_p$, it is easy to see from the graphical representation that $C_t$ is self-dual in the sense that
\begin{equation}
\label{paper1:dualeq}
\P_p [\, C_t^{\pi_p,A}\cap B \neq \emptyset\,]=\P_p [\, C_t^{\pi_p,B}\cap A \neq \emptyset\,] \quad \forall t > 0,\: A, B \subseteq \Zd.
\end{equation}
If we take $A=\{ 0\}$, $B=\Zd$ in this equation and let $t \to \infty$ using the previous lemma, we can easily conclude that 
\[
\P_p [\, C_t^{\pi_p,\{0\}} \neq \emptyset \, \forall t > 0 \,] > 0 \quad  \iff  \quad \bar{\nu}_p  \neq \pi_p\times\delta_{\emptyset}
\]
and we are done.
\end{proofoft} 

\noindent\Remark There is a weaker duality equation when the initial distribution of the background process differs from $\pi_p$, but this is less natural and seems less useful.   

\section{Proof of Theorem \ref{paper1:thm3}}

We now turn to the proof of Theorem~\ref{paper1:thm3}, that the critical CPREE dies out. Once Lemma~\ref{paper1:lemma1} below is established, the rest follows similar lines as in the proofs of Theorem~\ref{paper1:GB} carried out in \cite{GrimBez} and \cite{Liggett99}. Our main goal is to prove that if $\{C_t\}$ survives at $p>0$, then there is a number $\delta>0$ and integers $n,a$ such that
\begin{equation}\label{paper1:goal}
\P_{p-\delta}[\,C_t^{\emptyset,[-n,n]^d}\text{ survives in }\Z\times [-5a,5a]^{d-1}\times [0,\infty) \,]>0.
\end{equation} 
If $p_c\in(0,1]$, this will immediately imply 
\[
\P_{p_c}[\,C_t^{\emptyset,\{0\}}\neq\emptyset\:\forall t\geq 0 \,]=0.
\]
To achieve \eqref{paper1:goal}, we begin by showing that if the CPREE survives, then it is very likely to have survival if the initial configuration is sufficiently large even if we start with all zeros in the background process.
\begin{lemma}\label{paper1:lemma1}
If $\{C_t\}$ survives at $p>0$ then 
\[
\lim_{n\to\infty} \P_p [\,C_t^{\emptyset,[-n,n]^d}\neq\emptyset\,\forall t>0 \,]=1.
\]
\end{lemma} 
For the proof of this we use the following result.
\begin{lemma}
For all $n\geq 1$, we have
\[
\lim_{\epsilon\to 0}\P_p[\,C_t^{\pi_{p-\epsilon},[-n,n]^d}\neq\emptyset\,\forall t>0\,]=\P_p[\,C_t^{\pi_p,[-n,n]^d}\neq\emptyset\,\forall t>0\,].
\]
\end{lemma}
\begin{proof}
Fix $n\geq 1$. The probability on the left increases when $\epsilon$ decreases and so the limit exists and is clearly at most the right hand side. For the other inequality let $\delta>0$ and define
\[
\varphi_t^{\epsilon}(x)=1_{\{B_t^{\pi_{p-\epsilon}}(x)=B_t^{\pi_p}(x)\}} \quad x\in \Zd,\: t\geq 0,
\]
where $\pi_{p-\epsilon}$ and $\pi_p$ are coupled in the usual monotone way. Recall the definition of $\varphi_t$ from the proof of Theorem~\ref{paper1:thm1} and observe that 
\[
\varphi_t \subseteq\varphi_t^{\epsilon}\quad\forall t>0, \:\forall\epsilon>0. 
\]
Also, an easy modification of the proof of Lemma~\ref{paper1:BN} yields
\[
\lim_{T\to\infty}\P_p[\,\tilde{C}_t^{[-n,n]^d}\subseteq\varphi_t \:, \forall t\geq T\,]=1.
\]
(Recall that $\tilde{C}_t^A$ is the CPREE starting from the configuration $A$ but with no recoveries.) This allows us to choose $T>0$ such that
\[
\begin{split}
&\P_p[\,C_t^{\pi_p,[-n,n]^d}\neq\emptyset\,\forall t>0 \,] \\
&\quad\leq\P_p[\,\tilde{C}_t^{[-n,n]^d}\subseteq\varphi_t \:, \forall t\geq T,\,C_t^{\pi_p,[-n,n]^d}\neq\emptyset\,\forall t>0\,]+\delta.
\end{split}
\]
Given this $T$, choose $m\geq 1$ such that
\[
\P_p[\,\tilde{C}_t^{[-n,n]^d}\subseteq [-m,m]^d \: \forall\: 0\leq t\leq T\,]>1-\delta
\]
and for that $m$ choose $\epsilon_0>0$ such that
\[
\P_p[\,B_0^{\pi_{p-\epsilon}}=B_0^{\pi_p} \text{ on }[-m,m]^d\,] > 1-\delta,\quad \forall\, 0<\epsilon\leq\epsilon_0.
\]
Now since
\[
\begin{split}
 &\{\,\tilde{C}_t^{[-n,n]^d}\subseteq\varphi_t \:, \forall t\geq T,\,\tilde{C}_t^{[-n,n]^d}\subseteq [-m,m]^d \: \forall\: 0\leq t\leq T,\\
&\quad\,B_0^{\pi_{p-\epsilon}}=B_0^{\pi_p} \text{ on }[-m,m]^d,\,C_t^{\pi_p,[-n,n]^d}\neq\emptyset\,\forall t>0 \,\} \\
&\quad\subseteq \{\,C_t^{\pi_{p-\epsilon},[-n,n]^d}\neq\emptyset\,\forall t>0\,\},
\end{split}
\]
we get
\[
\begin{split}
&\P_p[\,C_t^{\pi_p,[-n,n]^d}\neq\emptyset\,\forall t>0 \,] \\
&\quad\leq\P_p[\,C_t^{\pi_{p-\epsilon},[-n,n]^d}\neq\emptyset\,\forall t>0 \,]+3\delta,
\end{split}
\]
whenever $0<\epsilon\leq\epsilon_0$ and so the proof is complete.
\end{proof} 
\noindent\begin{proofofl}{\ref{paper1:lemma1}}
Let $\delta>0$. From the self-duality equation \eqref{paper1:dualeq}, Lemma \ref{paper1:upper_conv1} and the easily verified fact that the second marginal of $\bar{\nu}_p$ gives zero measure to \small$\emptyset$\normalsize, we easily get that there is an $n\geq 1$ such that
\[
\P_p[\,C_t^{\pi_p,[-n,n]^d}\neq\emptyset\,\forall t>0 \,]>1-\delta.
\]
The previous lemma makes it possible to now choose an $\epsilon>0$ such that
\[
\P_p[\,C_t^{\pi_{p-\epsilon},[-n,n]^d}\neq\emptyset\,\forall t>0 \,]>1-\delta.
\]
Denote the semigroup operator associated with the background process by $T(t)$ and note that for $\epsilon$ above there is a time $s$ such that
\[
\delta_\emptyset T(s) \geq \pi_{p-\epsilon}.
\]
Now, let $B_{m,n}$ denote the box in $\Zd$ with sidelength $m n$ and write
\[
B_{m,n}=\bigcup_{i=1}^{m^d} A_i,
\]
where each $A_i$ is a translation of the box with sidelength $n$ and with the $A_i$'s disjoint. Then, define
\[
A_{m,n}^s=\{\,\text{No arrivals in }N^{\delta_1}\text{ or }N^{\delta_0-\delta_1}\text{ up to time $s$ in some }A_i \,\}.
\]
Given $n$ and $s$, we can choose $m$ so large that
\[
\P_p[\,A_{m,n}^s \,]>1-\delta.
\]
The proof is finished by noting that monotonicity easily implies that
\[
\P_p [\,C_t^{\emptyset,[-mn,mn]^d}\neq\emptyset\,\forall t>0\,|\,A_{m,n}^s \,]\geq\P_p[\,C_t^{\pi_{p-\epsilon},[-n,n]^d}\neq\emptyset\,\forall t>0 \,],
\]
using the fact that $A_{m,n}^s$ is independent of the background process.
\end{proofofl}

\noindent\Remark A slightly more abstract but considerably shorter proof of Lemma \ref{paper1:lemma1} is found by Olle Häggström after submission
of the paper and is as follows. For $x\in \Zd$, let $Y_x^{\emptyset}$ be the indicator variable for survival when the process starts with only $x$ infected and all zeros in the background process. By translation invariance, $\P_p[\,Y_x^{\emptyset}=1 \,]$ is independent of $x$ and by Theorem \ref{paper1:thm1} we know that it is positive. It follows from the graphical representation that the process $\{Y_x^{\emptyset}\}_{x\in \Zd}$ is ergodic and hence a.s.\ there is some $x$ for which $Y_x^{\emptyset}=1$. Moreover, the event in Lemma \ref{paper1:lemma1} occurs as soon as some site in $[-n,n]^d$ has $Y_x^{\emptyset}=1$ and so the lemma follows at once. 

We have now set up the necessary ground work for our model in order to be able to follow the steps in \cite{Liggett99}. For $L\geq 1$ and $A\subseteq (-L,L)^d$, let $\CL_t^{\emptyset,A}$ be the truncated process,  using only \small$\emptyset$\normalsize-active paths (recall Definition~\ref{paper1:paths}) which stay in $(-L,L)^d\times [0,t]$. 
\begin{lemma}\label{paper1:lemma2}
For all finite $A\subseteq\Zd$ and $N\geq 1$, we have
\[
\lim_{t\to\infty}\lim_{L\to\infty}\P_p[\,|\CL_t^{\emptyset,A}|\geq N \,]=\P_p[\,C_t^{\emptyset,A}\neq\emptyset\,\forall t >0 \,]
\]
\end{lemma}
\begin{proof}
Fix $A$ and $N$. Since 
\[
C_t^{\emptyset,A}=\bigcup_{L=1}^\infty\CL_t^{\emptyset,A},
\] 
we easily get that for fixed $t$
\[
\P_p[\,|C_t^{\emptyset,A}|\geq N \,]=\lim_{L\to\infty} \P_p[\,|\CL_t^{\emptyset,A}|\geq N \,],
\]
and so we are done if
\[
\lim_{t\to\infty}\P_p[\,|C_t^{\emptyset,A}|\geq N \,]=\P_p[\,C_t^{\emptyset,A}\neq\emptyset\,\forall t >0 \,].
\]
For this, it is enough to check two things:
\begin{align*}
\lim_{t\to\infty}&\P_p[\,|C_t^{\emptyset,A}|\geq N,C_s^{\emptyset,A}=\emptyset\,\text{ some } s >0 \,]=0 \\
\lim_{t\to\infty}&\P_p[\,|C_t^{\emptyset,A}|\geq N,C_s^{\emptyset,A}\neq\emptyset\,\forall s >0 \,]=\P_p[\,C_t^{\emptyset,A}\neq\emptyset\,\forall t >0 \,]
\end{align*}
The first equality follows easily by applying Fatou's Lemma. The second one follows if
\[
\lim_{t\to\infty}|C_t^{\emptyset,A}|=\infty \quad \text{a.s} \quad \text{on} \quad \{\,C_t^{\emptyset,A}\neq\emptyset\,\forall t >0 \,\}.
\]
Assume the contrary, i.e.
\begin{equation}\label{paper1:contr}
\P_p[\,|C_t^{\emptyset,A}|\text{ does not converges to infinity, } C_s^{\emptyset,A}\neq\emptyset\,\forall s >0 \,]>0.
\end{equation}
From the martingale convergence theorem we get that 
\begin{equation}\label{paper1:martconv}
\P_p[\,C_t^{\emptyset,A}\neq\emptyset\,\forall t \geq s\, |\, \mathcal F_s \,]\to 1_{\{\,C_t^{\emptyset,A}\neq\emptyset\,\forall t >0\,\}}\quad \text{ as } s \to\infty,
\end{equation}
where $\mathcal F_s$ is the $\sigma$-algebra generated by the whole process up to time $s$. Equation \eqref{paper1:contr} and \eqref{paper1:martconv} implies that with positive probability the following can happen:
\begin{align*}
&\lim_{s\to\infty}\P^{(\beta_s,C_s)}[\,C_t\neq\emptyset\,\forall t>0 \,]=1 \\
&\exists \: M>0, \: \{\tau_i\}_{i\geq 1} \: \ni \: \tau_1<\tau_2< \ldots <\tau_i \to \infty\,, |C_{\tau_i}|\leq M \: \forall i.
\end{align*}
However, using elementary facts about exponentially distributed variables, we get
\[
\begin{split}
&\P^{(\beta_{\tau_i},C_{\tau_i})}[\,C_t=\emptyset \text{ some } t>0 \,] \\
&\quad\geq \P^{(\Zd,C_{\tau_i})}[\,C_t=\emptyset \text{ some } t>0 \,] \geq\left(\frac{\delta_1}{\delta_0+\gamma+2d}\right)^M \quad\forall i,
\end{split}
\]
which yields a contradiction and the proof is complete.
\end{proof} 
The next step is to take care of the sides of the space-time box. Define
\[
S(L,T)=\{\,(x,t)\in\Zd\times [0,T]:\,|x|_{\infty}=L \,\}.
\]
Fix $A\subseteq (-L,L)^d$ and look at all points on $S(L,T)$ that can be reached from $A$ by an \small$\emptyset$\normalsize-active path using vertical segments where the space coordinate is in $(-L,L)^d$ and infection arrows from $(x,\cdot)$ to $(y,\cdot)$ with $x\in (-L,L)^d$. Define $N_\emptyset^A(L,T)$ to be the maximum number of such points with the following property: If $(x,t_1)$ and $(x,t_2)$ are any two points with the same spatial coordinate, then $|t_1-t_2|\geq 1$.
\begin{lemma}\label{paper1:lemma3}
Assume $L_j \nearrow \infty$ and $T_j \nearrow \infty$. Then for any $M,N \geq 1$ and finite $A\subseteq\Zd$, we have
\[
\limsup_{j\to\infty} \P_p[\,N_\emptyset^A(L_j,T_j)\leq M \,]\,\P_p[\,|{}_{L_j}\,{}^{\!\!} C_{T_j}^{\emptyset,A}|\leq N\,] \leq \P_p[\,C_t^{\emptyset,A}=\emptyset\text{ some }t>0 \,].
\]
\end{lemma}
\begin{proof}
The proof follows the steps of Proposition 2.8 in \cite{Liggett99} with some adjustments. Let $\mathcal F_{L,T}$ denote the $\sigma$-algebra generated by $M^{b,0\to 1}$, $M^{b,1\to 0}$, $N^{\delta_1}$, $N^{\delta_0-\delta_1}$ and $\vec{N}^j$, $j\in \{\,\pm e_1,\ldots,\pm e_d \,\}$  in $(-L,L)^d \times [0,T]$. We first argue that
\begin{equation}\label{paper1:cond1}
\begin{split}
&\P_p[\,C_t^{\emptyset,A}=\emptyset\text{ some }t>0\,|\,\mathcal F_{L,T} \,]\geq\left(\frac{e^{-4d}\delta_1}{\delta_0+\gamma+2d}\right)^k \\ 
&\quad\text{a.s}\quad\text{on}\quad \{\,N_\emptyset^A(L,T)+|\CL_T^{\emptyset,A}|\leq k \,\}
\end{split}
\end{equation}
For $x\in\CL_T^{\emptyset,A}$ there is a conditional probability of at least
\[
\frac{\delta_1}{\delta_0+\gamma+2d} 
\]
that $x$ becomes healthy before it infects any of its neighbors. So, if $|\CL_T^{\emptyset,A}|=m$, then the conditional probability that no $x\in\CL_T^{\emptyset,A}$ contributes to survival is at least
\[
\left(\frac{\delta_1}{\delta_0+\gamma+2d}\right)^m. 
\]
For the sides of the box, consider a time line $\{x\}\times [0,T]$, where $|x|_{\infty}=L$ and let
\[
(x,t_1),\ldots,(x,t_j)
\]
be a maximal set of points that can be reached from $A$ by an \small$\emptyset$\normalsize-active path with the property that each pair is separated by at least distance 1. Let
\[
I=\bigcup_{k=1}^j \{x\}\times(t_k-1,t_k+1)
\]
and note that the probability that there are no arrows coming out from $I$ is at least $e^{-4dj}$. Furthermore, for each interval of length $y$ in the complement of $I$ in $\{x\}\times [0,\infty)$, the probability of the event that if there is at least one arrival of the Poisson processes in the interval with the first one coming from $N^{\delta_1}$ or there is no arrivals at all is 
\[
\left(1-e^{-(\delta_0+\gamma+2d)y}\right)\frac{\delta_1}{\delta_0+\gamma+2d}+e^{-(\delta_0+\gamma+2d)y}\geq \frac{\delta_1}{\delta_0+\gamma+2d}.
\]
By independence, we get that the conditional probability that none of the points in the time line $\{x\}\times [0,T]$ contributes to survival is at least
\[
\left(\frac{e^{-4d}\delta_1}{\delta_0+\gamma+2d}\right)^j.
\]
Now, considering the contribution of different $x$'s yields
\[
\begin{split}
&\P_p[\,C_t^{\emptyset,A}=\emptyset\text{ some }t>0\,|\,\mathcal F_{L,T} \,] \\
&\quad\geq\left(\frac{\delta_1}{\delta_0+\gamma+2d}\right)^{|\CL_T^{\emptyset,A}|}\left(\frac{e^{-4d}\delta_1}{\delta_0+\gamma+2d}\right)^{N^A(L,T)} 
\end{split}
\]
which implies \eqref{paper1:cond1}. For the rest of the proof, one proceeds exactly as in the second half of Proposition 2.8 in \cite[p.~48-49]{Liggett99}. The needed inequality
\[
\begin{split}
&\P_p[\,N_\emptyset^A(L,T)\leq M,|\CL_T^{\emptyset,A}|\leq N \,] \\
&\quad\geq \P_p[\,N_\emptyset^A(L,T)\leq M\,]\P_p[\,|\CL_T^{\emptyset,A}|\leq N \,]
\end{split}
\]
is justified by the fact that $N_\emptyset^A(L,T)$ and $|\CL_T^{\emptyset,A}|$ are increasing functions of $\vec{N}^j$, $j\in \{\,\pm e_1,\ldots,\pm e_d \,\}$ and $M^{b,0\to 1}$, and decreasing in $N^{\delta_1}$, $N^{\delta_0-\delta_1}$ and $M^{b,1\to 0}$. This completes the proof.  
\end{proof} 
We are soon ready to state and prove the so called finite space-time condition. However, we first need two more propositions. We just state them here since the proofs are exactly the same as for Propositions 2.6 and 2.11, pages 46-47 and 49 in \cite{Liggett99}. 
\begin{prop}\label{paper1:prop1}
For every $n,N\geq 1$ and $L\geq n$, we have
\[
\P_p[\,|\CL_t^{\emptyset,[-n,n]^d}\cap[0,L)^d|\leq N \,]\leq \left(\P_p[\,|\CL_t^{\emptyset,[-n,n]^d}|\leq 2^d N \,]\right)^{2^{-d}}
\]
\end{prop}
Let 
\[
S_+(L,T)=\{\,(x,t)\in\Zd\times [0,T]:\,x_1=L \,,x_i\geq 0,\,2\leq i \leq d \,\}
\]
and define $N_{\emptyset,+}^A(L,T)$ in a similar manner as $N_\emptyset^A(L,T)$ using $S_+(L,T)$ instead of $S(L,T)$.
\begin{prop}\label{paper1:prop2}
For any $L,M \geq 1$, $T>0$ and $n < L$,
\[
\left(\P_p[\,N_{\emptyset,+}^{[-n,n]^d}(L,T)\leq M \,]\right)^{d2^d} \leq \P_p[\,N_{\emptyset}^{[-n,n]^d}(L,T)\leq Md2^d \,]
\]
\end{prop}
The proof of these propositions requires certain random variables to be positively correlated. For Proposition \ref{paper1:prop1}, let $X_1=|\CL_t^{[-n,n]^d}\cap [0,L)^d|$ and $X_2,\ldots,X_{2^d}$ be defined similarly with respect to the other orthants in $\RR^d$. The needed positive correlation of $\{X_i\}_{i=1}^{2^d}$ is justified in the same way as in the end of the proof of Lemma \ref{paper1:lemma3}. Similarly justification can be made in the proof of Proposition \ref{paper1:prop2}. 
\begin{theorem}
If $\{C_t\}$ survives at $p>0$, then it satisfies the following condition: For all $\epsilon>0$ there exist $n,L \geq 1$ and $T>0$ such that
\begin{align}
&\P_p[\,{}_{L+n}\dptn C_{T+1}^{\emptyset,[-n,n]^d}\supseteq \:x+[-n,n]^d \text{ some } x\in[0,L)^d \,]>1-\epsilon \label{paper1:fstc1} \\ 
&\P_p[\,{}_{L+2n+1}\dptn C_{t+1}^{\emptyset,[-n,n]^d}\supseteq \:x+[-n,n]^d \text{ some } 0\leq t<T,\label{paper1:fstc2} \\
&\quad\text{some } x\in\{L+n\}\times [0,L)^{d-1} \,]>1-\epsilon \nonumber
\end{align}
\end{theorem}
\begin{proof}
Again, we will follow the steps in \cite{Liggett99} with some modifications. Let $0<\delta<1$. We will see at the end how to choose $\delta$ for a given $\epsilon>0$. Lemma \ref{paper1:lemma1} gives us an $n$ such that
\begin{equation}\label{paper1:step1}
\P_p[\,C_t^{\emptyset,[-n,n]^d}\neq\emptyset\,\forall t>0 \,]>1-\delta^2.
\end{equation}
Given $n$, choose $N^\prime$ such that
\[
\left(1-\P_p[\,{}_{n+1}\dptn C_1^{\emptyset,\{0\}}\supseteq [-n,n]^d \,]\right)^{N^\prime} < \delta
\] 
and then choose $N$ so large such that if $A\subseteq\Zd$ with $|A|\geq N$, then there exists $B\subseteq A$ with $|B|\geq N^\prime$ and 
\[
|x-y|_\infty \geq 2n+1 \quad \forall\,x,y \in B,\: x\neq y.
\]
Let $B_A$ be a fixed (deterministic) such choice for each $A$. 

In a similar fashion, choose $M^\prime$ such that
\begin{equation}\label{paper1:Mprime}
\left(1-a\right)^{M^\prime} < \delta,
\end{equation}
where
\[
\begin{split}
a=\P_p[\,&\text{There are \small$\emptyset$\normalsize-active paths from the origin to every} \\
&\text{point in } [0,2n]\times [-n,n]^{d-1}\times \{1\} \text{ that}\\
&\text{stays in } [0,2n]\times [-n,n]^{d-1}\times [0,1]\,]
\end{split}
\]
Then choose $M$ so large such that if $A\subseteq\Zd\times [0,\infty)$ is a finite set with $|A|\geq M$, where the distance in time between points with the same spatial coordinate is at least 1, then there exists $B\subseteq A$ with $|B|\geq M^\prime$ and with the property that for each pair of points $(x,s),(y,t)\in B$ we have either
\begin{equation}\label{paper1:sep}
x=y, \quad |s-t|\geq 1 \qquad \text{or} \qquad |x-y|_\infty \geq 2n+1.
\end{equation}
Let $B_A$ be a fixed (deterministic) such choice for each $A$.

From Lemma \ref{paper1:lemma2}, \eqref{paper1:step1}, the inequality $1-\delta<1-\delta^2$ and the facts that for fixed $L$, $n$ and $N$, the map $t \mapsto \P_p[\,|\CL_t^{\emptyset,[-n,n]^d}|>2^d N\,]$ is continuous and that $\lim_{t\to\infty}\P_p[\,|\CL_t^{\emptyset,[-n,n]^d}|>2^d N\,]=0$, we can conclude that there exist $L_j \nearrow \infty$ and $T_j \nearrow\infty$ so that
\[
\P_p[\,|{}_{L_j}\dptn C_{T_j}^{\emptyset,[-n,n]^d}|>2^d N \,]=1-\delta \quad \forall\, j\geq 1.
\]
Furthermore, Lemma \ref{paper1:lemma3} with $M$ and $N$ replaced by $Md2^d$ and $2^d N$ respectively and with $A=[-n,n]^d$, we get that for some $j$ 
\[
\P_p[\,N_\emptyset^{[-n,n]^d}(L_j,T_j)>Md2^d \,] > 1-\delta.
\]
Let $L=L_j$ and $T=T_j$ for that specific $j$ and apply Propositions \ref{paper1:prop1} and \ref{paper1:prop2} to get
\begin{align}\label{paper1:Cplus}
\P_p[\,|\CL_T^{\emptyset,[-n,n]^d}\cap[0,L)^d |> N\,]& >1-\delta^{2^{-d}} \\ \label{paper1:Nplus}
\P_p[\,N_{\emptyset,+}^{[-n,n]^d}(L,T)>M \,]& >1-\delta^{2^{-d}/d}.
\end{align}
To obtain \eqref{paper1:fstc1}, define for $B\subseteq\Zd$ and $T>0$
\[
\begin{split}
V_B^{T}&=\{\,\exists\, (x,t)\in B\times\{T\}\text{ such that there are \small$\emptyset$\normalsize-active paths from }\\
&\qquad(x,t)\text{ to every }(y,s)\in \left(x+[-n,n]^d \right)\times\{T+1\} \\ 
&\qquad\text{ that stays in } \left(x+[-n,n]^d \right)\times (T,T+1] \,\},
\end{split}
\]
and note that
\begin{equation}\label{paper1:incl}
\begin{split}
&\bigcup_{A\subseteq [0,L)^d} \{\,|\CL_T^{\emptyset,[-n,n]^d}\cap[0,L)^d |> N,\,\CL_T^{\emptyset,[-n,n]^d}\cap[0,L)^d=A,\,V_{B_A}^{T} \,\} \\
&\qquad\subseteq\{\,{}_{L+n}\dptn C_{T+1}^{\emptyset,[-n,n]^d}\supseteq \, x+[-n,n]^d \text{ some } x\in [0,L)^d\,\}. 
\end{split}
\end{equation}
Let $\mathcal F_T$ be the $\sigma$-algebra generated by $M^{b,0\to 1}$, $M^{b,1\to 0}$, $N^{\delta_1}$, $N^{\delta_0-\delta_1}$, and $\vec{N}^j$, $j\in \{\,\pm e_1,\ldots,\pm e_d \,\}$ up to time $T$ and note that for given $A\subseteq [0,L)^d$ with $|A|\geq N$, $V_{B_A}^T$ is independent of $\mathcal F_T$ so
\[
\begin{split}
&\P_p[\,V_{B_A}^{T}\,|\,\mathcal F_T \,]=\P_p[\,V_{B_A}^{T}\,] \\
&\quad\geq 1-\left(1-\P_p[\,{}_{n+1}\dptn C_1^{\emptyset,\{0\}}\supseteq [-n,n]^d \,]\right)^{N^\prime}>1-\delta.
\end{split}
\]
By summing up over $A\subseteq [0,L)^d$ and using \eqref{paper1:Cplus} and \eqref{paper1:incl}, we get
\[
\begin{split}
&\P_p[\,{}_{L+n}\dptn C_{T+1}^{\emptyset,[-n,n]^d}\supseteq \, x+[-n,n]^d \text{ some } x\in [0,L)^d \,] \\
&\quad > (1-\delta)(1-\delta^{2^{-d}}).
\end{split}
\]
This yields \eqref{paper1:fstc1} when $\delta$ is chosen appropriately.

To obtain \eqref{paper1:fstc2}, define for each space-time point $(x_i,t_i)$ we count in \small$N_{\emptyset,+}^{[-n,n]^d}(L,T)$ \normalsize a variable $\tilde{Y}_i$ which is 1 if $(x_i,t_i)$ infects all points in 
\[
\left(x_i+[0,2n]\times [-n,n]^{d-1}\right)\times\{t_i+1\}
\]
using \small$\emptyset$\normalsize-active paths in 
\[
\left(x_i+[0,2n]\times [-n,n]^{d-1}\right)\times (t_i,t_i+1]
\]
only and 0 otherwise. If $N_{\emptyset,+}^{[-n,n]^d}(L,T)>M$, we can choose $M^\prime$ space-time points satisfying \eqref{paper1:sep}. Denote the corresponding variables by $Y_i$, $i=1,\ldots,M^{\prime}$. Let $\mathcal F_{L,T}$ be as in the proof of Lemma \ref{paper1:lemma3} and note that conditioned on $\mathcal F_{L,T}$ restricted to the event $\{\,N_{\emptyset,+}^{[-n,n]^d}(L,T)>M\,\}$, the $M^\prime$ space-time points are specified and $Y_1,Y_2,\ldots,Y_{M^\prime}$ are independent with the (conditional) probability of $Y_i=1$ equal to $a$. This implies that
\[
\begin{split}
&\P_p[\,Y_i=1 \text{ some } i=1,\ldots,M^\prime\,|\,\mathcal F_{L,T} \,]= 1-\left(1-a\right)^{M^\prime} \\
&\qquad\text{on}\quad \{\,N_{\emptyset,+}^{[-n,n]^d}(L,T)>M \,\},
\end{split}
\]
which together with \eqref{paper1:Mprime} and \eqref{paper1:Nplus} yields
\[
\begin{split}
&\P_p[\,{}_{L+2n+1}\dptn C_{t+1}^{\emptyset,[-n,n]^d}\supseteq \:x+[-n,n]^d \text{ some } 0\leq t<T,\\
&\qquad\text{some } x\in\{L+n\}\times [0,L)^{d-1} \,] \\
&\qquad> (1-\delta)(1-\delta^{2^{-d}/d}).
\end{split}
\]
This gives \eqref{paper1:fstc2} when $\delta$ is chosen appropriately.
\end{proof} 
The next part of the program is to carry out a comparison with oriented percolation. For this, we start to combine \eqref{paper1:fstc1} and \eqref{paper1:fstc2} into one. 
\begin{lemma}\label{paper1:l1}
If $\{C_t\}$ survives at $p>0$, then it satisfies the following condition: For all $\epsilon>0$ there exist $n,L \geq 1$ and $T>0$ such that
\begin{equation}\label{paper1:fstc3}
\begin{split}
&\P_p[\, {}_{2L+3n}\dptn C_t^{\emptyset,[-n,n]^d} \supseteq \:x+[-n,n]^d \text{ some } T\leq t<2T, \\
&\qquad\text{some } x\in [L+n,2L+n]\times [0,2L)^{d-1} \,]>1-\epsilon 
\end{split}
\end{equation}
\end{lemma}
\begin{proof}
We follow Proposition 2.20 in \cite{Liggett99}. Let $(x,\tau)$ be the first (in time) space-time point with the property appearing in the probability \eqref{paper1:fstc2}, where $x$ is choosen according to some deterministic ordering of $\Zd$ and restart $(B_t,C_t)$ at time $\tau+1$. From \eqref{paper1:fstc1}, \eqref{paper1:fstc2} and the fact that these probabilities are increasing in the background process, it follows that
\[
\begin{split}
&\P_p[\, {}_{2L+3n}\dptn C_t^{\emptyset,[-n,n]^d} \supseteq \:x+[-n,n]^d \text{ some } T+1\leq t<2T+2, \\
&\qquad\text{some } x\in [L+n,2L+n]\times [0,2L)^{d-1} \,]>(1-\epsilon)^2. 
\end{split}
\]
Replace $T+1$ with $T$ and the proof is complete.
\end{proof} 
Now we are ready for the fundamental step in the construction towards the comparison. 
\begin{lemma}
Assume $\{C_t\}$ survives at $p>0$ and fix $\epsilon>0$. Then there exist $\delta>0$, $n,a,b$ with $n<a$ such that for all $(x,t) \in [-a,a]^d \times [0,b]$ 
\[
\begin{split}
&\P_{p-\delta} [\,\exists\, (y,s)\in [a,3a]\times [-a,a]^{d-1}\times [5b,6b]\text{ such that} \\
&\qquad\text{there are \small$\emptyset$\normalsize-active paths from } (x,t)+\left([-n,n]^d\times\{0\}\right) \\
&\qquad\text{to every point in }(y,s)+ \left([-n,n]^d\times\{0\} \right) \\ 
&\qquad\text{that stays in } [-5a,5a]^d \times [0,6b] \,]>1-\epsilon.
\end{split}
\]
\end{lemma} 
\begin{proof}
One can proceed exactly as in Proposition 2.22 in \cite[p.~52-53]{Liggett99} to first obtain the statement with $p-\delta$ replaced by $p$ and therefore we only outline this part of the argument. The main idea is to use Lemma~\ref{paper1:l1} (or a ``reflected'' version of it) repeatedly (between $4$ to $10$ times) to steer things properly so that the desired event occurs. The existence of $\delta>0$ is a consequence of the fact that the event in question depends only on the graphical representation in $[-5a,5a]^d\times [0,6b]$ and hence is continuous in $p$. 
\end{proof} 
Repeated use of the previous lemma together with appropriate stopping times and monotonicity in the background process yields:
\begin{lemma}\label{paper1:ren}
Assume $\{C_t\}$ survives at $p>0$ and let $\epsilon>0$ and $k \geq 1$ be fixed. Then there exist $\delta>0$, $n,a,b$ with $n<a$ such that the following holds: For all $(x,t) \in [-a,a]^d \times [0,b]$, with $\P_{p-\delta}$-probability at least $1-\epsilon$, there exists a translate $(y,s)+[-n,n]^d\times\{0\}$ of $[-n,n]^d\times\{0\}$ such that 
\begin{align*}
\text{a) } &(y,s)\in \left([-a,a]+2ka\right)\times [-a,a]^{d-1}\times [5kb,(5k+1)b] \\
\text{b) } &\text{There are \small$\emptyset$\normalsize-active paths from } (x,t)+[-n,n]^d\times\{0\} \text{ to every} \\ 
&\text{point in }(y,s)+[-n,n]^d\times\{0\} \text{ that stays in the region} \\
&\mathcal A =\bigcup_{j=0}^{k-1} \left([-5a,5a]+2ja\right)\times [-5a,5a]^{d-1} \times\left([0,6b]+5jb\right).
\end{align*}
\end{lemma}
\begin{figure}[h]
\begin{center}
\includegraphics[scale=0.55]{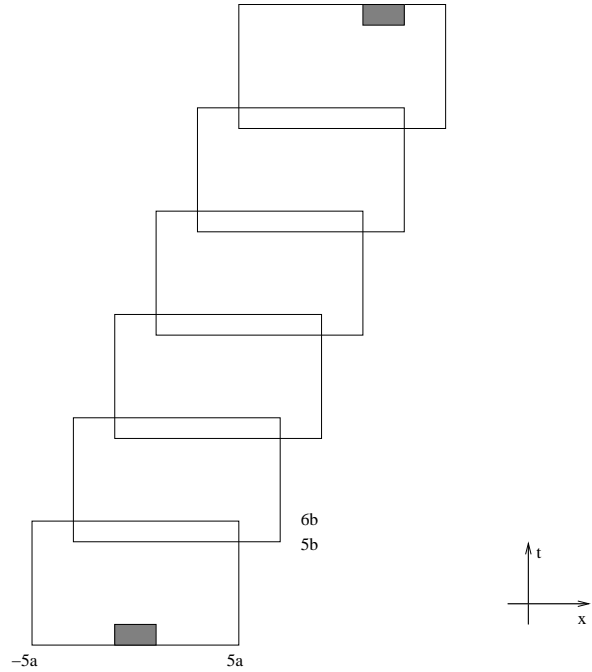}
\caption{The set $\mathcal A$.}
\label{paper1:block1}
\end{center}
\end{figure}
Our final step towards \eqref{paper1:goal} is to use the previous lemma in a so called renormalization argument. The set $\mathcal A$ from Lemma \ref{paper1:ren} (see Figure~\ref{paper1:block1}) and its reflection with respect to the $t$-axis will consist of our building blocks. Given the conditions in Lemma~\ref{paper1:ren}, the distance c in Figure~\ref{paper1:block2} is well defined. (Define it to be zero if the dashed vertical line is to the right of the left corner of the rectangle $R$, see Figure~\ref{paper1:block2}.) It is easy to see that, if we choose $k>5$, $c$ will be bigger than $3a$, independent of the value of $a$. Fix such a $k$.
\begin{figure}[h]
\begin{center}
\includegraphics[scale=0.6]{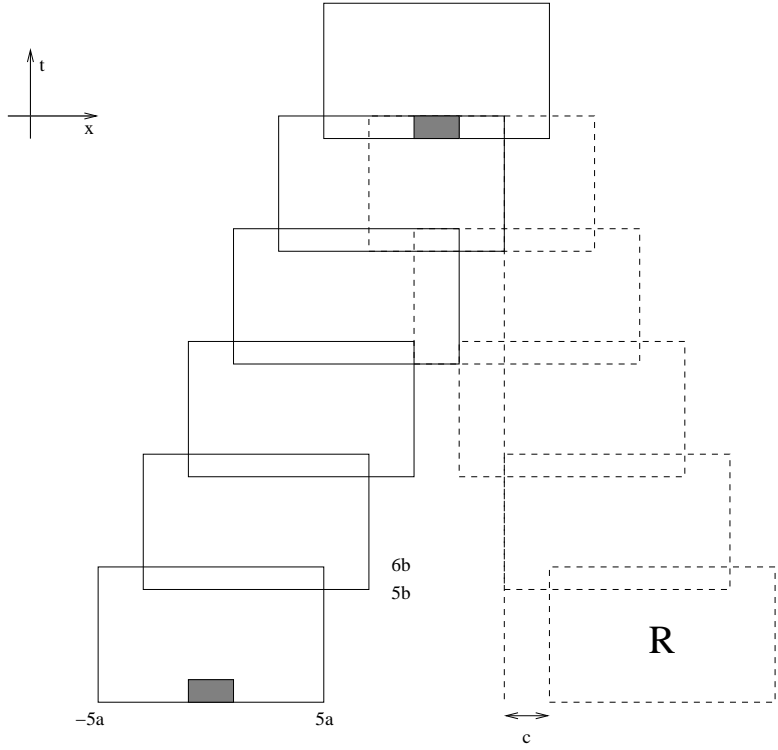}
\caption{The definition of $c$.}
\label{paper1:block2}
\end{center}
\end{figure}
\begin{figure}[h]
\begin{center}
\includegraphics[scale=0.5]{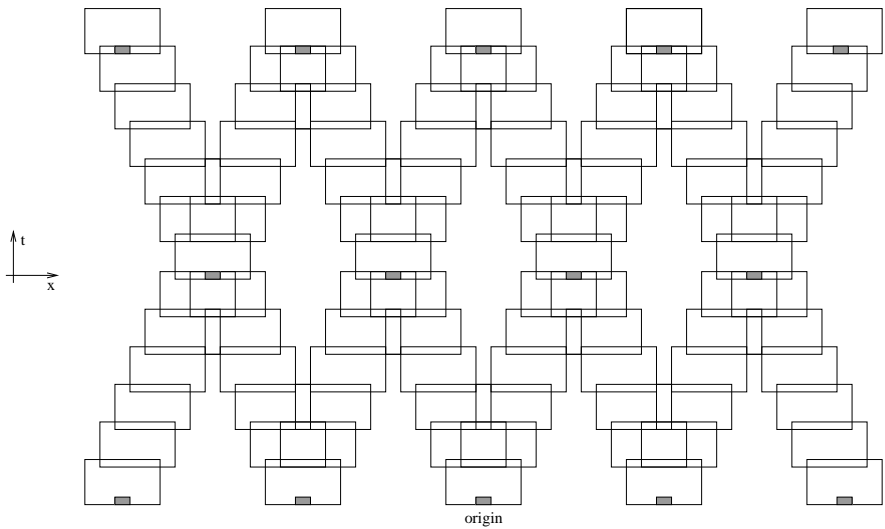}
\caption{Our building block $\mathcal A$ together with its reflection are translated in the $x_1$ and $t$ direction. The shaded regions indicate where the paths start and stop in the definition of $Z_n$.}
\label{paper1:block3}
\end{center}
\end{figure}
\begin{theorem}
If $\{C_t\}$ survives at $p>0$, then there are integers $n$,$a$ and $\delta>0$ such that 
\[
\P_{p-\delta}[\,C_t^{\emptyset,[-n,n]^d}\text{ survives in }\Z\times [-5a,5a]^{d-1}\times [0,\infty) \,]>0
\]
\end{theorem}
\begin{proof}
The proof is a modification of Lemma 21 of \cite{GrimBez}. Let $\eta>0$ be given and take $\epsilon>0$ such that $1-\epsilon>1-\eta$ and let $n$, $a$, $b$ and $\delta$ be as in Lemma~\ref{paper1:ren}. We will make an appropriate choice of $\eta$ later. Construct a process $Z_n(i)=\left(X_n(i),Y_n(i)\right)$, $i\geq 0$, $n\geq 0$, where $X_n(i)\in \{0,1\}$ and $Y_n(i)$ is a point in $\Zd\times[0,\infty)$. $Y_n(i)$ will be undefined when $X_n(i)=0$. Start with $Z_0(0)=(1,0)$, $X_0(i)=0$, $i\neq 0$ and define inductively as follows: With $Z_k(i)$ already defined for $i\geq 0$, $0\leq k\leq n$ let $X_{n+1}(i)=1$ if for either $j=i$ or $j=i-1$ it is the case that $X_n(j)=1$ and there is a translation of $[-n,n]^d$ to the shaded area (see Figure~\ref{paper1:block3} for the shaded regions) on the top of the corresponding block such that $Y_n(j)+[-n,n]^d$ is connected with \small$\emptyset$\normalsize-active paths to every point in that translation. Furthermore, define $Y_{n+1}(i)=(x_{n+1}(i),t_{n+1}(i))$, where $t_{n+1}(i)$ is the earliest center of such a translation and $x_{n+1}(i)$ is chosen according to some fixed ordering of $\Zd$. Note that if $X_n(i)=1$ for infinitely many pairs $(i,n)$, then $C_t^{\emptyset,[-n,n]^d}\text{ survives in }\Z\times [-5a,5a]^{d-1}\times [0,\infty)$ so it remains to prove that the former has positive probability. Let $\mathcal F_n$ be the $\sigma$-algebra generating by $Z_k(i)$, where $i\geq 0$, $0\leq k\leq n$ and note that from Lemma~\ref{paper1:ren} we get
\[
\P_p[\,X_{n+1}(i)=1\,|\,\mathcal F_n \,] > 1-\eta \quad\text{on}\quad \{\,X_n(i-1)=1 \text{ or } X_n(i)=1\,\}.
\]
Also, our choice of $k$ and the fact that events that depend on disjoint parts of the graphical representation are independent, we have that, conditioned on $\mathcal F_n$, the collection of variables $\{\,X_{n+1}(i):\: i\geq 0\,\}$ is one-dependent. Now, we are ready to make the construction above for a specific choice of $\eta$. Take $1/4 \leq p<1$ so large that an oriented percolation process, $\{A_n\}$, on $\N$ with parameter $p$ survives with positive probability when it starts with a single infection at the origin and choose $\eta$ such that $1-\eta > 1-(1-\sqrt{p})^{3}$. A result of Liggett, Schonmann and Stacey \cite{Domination} (see also Theorem B26 \cite{Liggett99}) tells us that a one-dependent process with density $1-\eta$ stochastically dominates a product measure with density $p$ on $\N$. We can then conclude that $\{X_n\}$ dominates $\{A_n\}$. This completes the proof.
\end{proof}    
\noindent 
We end with the following question:

\medskip

\noindent 
Does the process obey a complete convergence theorem, i.e.$\dptn$ is it the case that for all $p\in [0,1]$ and $\beta$, $\eta\in\{0,1\}^{\Zd}$
\[
(\delta_\beta\times\delta_\eta)S_p(t)\to \alpha_p(\beta,\eta)\bar{\nu}_p+(1-\alpha_p(\beta,\eta))\pi_p\times \delta_{\emptyset}\quad \text{as $t\to\infty$},
\]
where 
\[
\alpha_p(\beta,\eta)=\P_p[\,C_t^{\beta,\eta}\neq\emptyset\,\forall t\geq 0\,]. 
\]
\noindent
Contemporaneously and independently of our work, Remenik \cite{Rem} has proved a complete convergence theorem for the special variant when $\delta_0=\infty$. We strongly believe that a complete convergence theorem also holds in our case and plan to pursue some ideas that we have.

\section*{Acknowledgement}

The authors want to thank Olle Häggström for a careful reading of the manuscript and for valuable comments.

%--------------------------------------------------------------------------------------

\bibliographystyle{amsplain}
%\bibliography{/home/warfheimer/Backup_remote/Forskning/Paper1/references}
%\bibliography{/chalmers/users/warfheim/Forskning/Paper1/references}
%\bibliography{/home/warfheim/Backup_remote/Forskning/Paper1/references}

\begin{thebibliography}{10}

\bibitem{GrimBez}
C.~Bezuidenhout and G.~Grimmett, \emph{The critical contact process dies out},
  Ann. Probab. \textbf{18} (1990), 1462--1482.

\bibitem{Durrett1}
M.~Bramson, R.~Durrett, and R.~H. Schonmann, \emph{The contact process in a
  random environment}, Ann. Probab. \textbf{19} (1991), 960--983.

\bibitem{Broman}
E.~I. Broman, \emph{Stochastic {D}omination for a {H}idden {M}arkov {C}hain
  with {A}pplications to the {C}ontact {P}rocess in a {R}andomly {E}volving
  {E}nvironment}, Ann. Probab. \textbf{35} (2007), 2263--2293.

\bibitem{Durrett88}
R.~Durrett, \emph{Lecture notes on paricle systems and percolation}, Wadsworth
  and Brooks/Cole Advances Books and Software, (1988).

\bibitem{Harris1}
T.~E. Harris, \emph{Contact interaction on a lattice}, Ann. Probab. \textbf{2}
  (1974), 969--988.

\bibitem{Extinction1}
A.~Klein, \emph{Extinction of contact and percolation processes in a random
  envionment}, Ann. Probab. \textbf{22} (1994), 1227--1251.

\bibitem{Liggett85}
T.~M. Liggett, \emph{Interacting {P}article {S}ystems}, Springer, (1985).

\bibitem{Liggett1}
\bysame, \emph{The survival of one-dimensional contact processes in random
  environments}, Ann. Probab. \textbf{20} (1992), 696--723.

\bibitem{Liggett99}
\bysame, \emph{Stochastic interacting systems: contact, voter and exclusion
  processes}, Springer, (1999).

\bibitem{Domination}
T.~M. Liggett, R.~H. Schonmann, and A.~M Stacey, \emph{Domination by product
  measures}, Ann. Probab. \textbf{25} (1997), 71--95.

\bibitem{Survival1}
C.~M. Newman and S.~B. Volchan, \emph{Persistent survival of one-dimensional
  contact processes in random environments}, Ann. Probab. \textbf{24} (1996),
  411--421.

\bibitem{Rem}
D.~Remenik, \emph{The contact process in a dynamic random environment}, Ann.
  Appl. Probab. \textbf{18} (2008), 2392--2420.
\end{thebibliography}

\end{document}